\documentclass[12pt]{amsart}

\usepackage{amsmath,amssymb}
\usepackage{amsfonts}
\usepackage{amsthm}
\usepackage{latexsym}
\usepackage{graphicx}
\usepackage{epsfig}

\newtheorem{theorem}{Theorem}[section]
\newtheorem{lemma}[theorem]{Lemma}
\newtheorem{proposition}[theorem]{Proposition}
\newtheorem{definition}[theorem]{Definition}

\newtheorem{conjecture}{Conjecture}

\newtheorem{prop}{Proposition}[section]
\newtheorem{remark}[prop]{Remark}

\makeatletter \@addtoreset{equation}{section} \makeatother

\def\Rc{{\mathrm {Rc}}}

\begin{document}

\title{Complete manifolds with bounded curvature and 
spectral gaps}

\author{Richard Schoen}
\author{Hung Tran}
\address{Department of Mathematics,
University of California, Irvine,
Irvine, CA 92697}
\thanks{The first named author was partially supported by NSF grant DMS-1404966}



\renewcommand{\subjclassname}{%
  \textup{2000} Mathematics Subject Classification}

\date{\today}

\begin{abstract} We study the spectrum of complete noncompact manifolds
with bounded curvature and positive injectivity radius. We give general
conditions which imply that their essential spectrum has an arbitrarily large finite number of gaps.
In particular, for any noncompact covering of a compact manifold, there is a metric
on the base so that the lifted metric has an arbitrarily large finite number of gaps in its
essential spectrum. Also, for any complete noncompact manifold with bounded curvature
and positive injectivity radius we construct a metric uniformly equivalent to the
given one (also of bounded curvature and positive injectivity radius) with an arbitrarily
large finite number of gaps in its essential spectrum.
\end{abstract}
\maketitle
\tableofcontents

\section{\textbf{Introduction}}
On a complete manifold, the Laplacian $\Delta$ acts as a self-adjoint operator on the space of smooth functions with compact support $C_c^\infty (M)$. There is a unique maximal self-adjoint extension to $L^2(M)$. Unlike the compact case, noncompact manifolds generally do not have pure point spectrum;
that is eigenvalues of finite multiplicity. For the example of $\mathbb{R}^n$ with a rotationally symmetric metric, see \cite{escobar86}. 
\begin{definition}
The essential spectrum, $\sigma_\text{ess}(M)$, is defined to be the set of real numbers which are either cluster points of the spectrum of $\Delta$ or eigenvalues with infinite multiplicity. 
\end{definition}
 
The essential spectrum turns out to be stable under compactly supported perturbations of the metric and, thus, is a function of the `ends' of $(M,g)$ (fundamental decomposition principle).  

In general terms the spectrum has been understood for complete manifolds with nonnegative
Ricci curvature. In fact, it was shown by J. Wang \cite{wang97} that if $\Rc(x)\geq -\delta(n)\frac{1}{r^2}$, for large r, and a small constant $\delta(n)$ depending on the dimension $n$, then the essential spectrum is $[0,\infty)$. In fact Wang \cite{wang97} shows, under his assumption on the Ricci curvature, that the $L^p$ essential spectrum is $[0,\infty)$ for all $p\geq 1$. That relies on work of K. T. Sturm \cite{sturm93} who showed the following.
\begin{theorem} \cite{sturm93}
Let M be a complete, non-compact Riemannian manifold with Ricci curvature bounded below and having volume growth uniformly sub-exponential. Then the $L^p$ essential spectra are the same for any $p\geq 1$.
\end{theorem}

This work was extended by Z. Lu and D. Zhou \cite{luzhou11}
who proved that the $L^p$ essential spectrum is $[0,\infty)$ under the assumption that
$\lim_{x\to\infty}Rc(x)\geq 0$. These authors proved the same result in case $M$ is a
gradient Ricci soliton with uniformly sub-exponential volume growth.

It has been less clear how the spectrum should behave for manifolds without the asymptotic
nonnegativity assumption on the Ricci curvature. Lu and Zhou made the following conjecture.
\begin{conjecture}\cite{luzhou11}
Let M be a complete, non-compact Riemannian manifold with Ricci curvature bounded below and the volume growth uniformly sub-exponential. Then, for any $p\geq 1$, the $L^p$-essential spectrum is $[0,\infty)$.
\end{conjecture}
In this paper we give general conditions on a complete manifold under which there are a
large finite number of gaps in the essential spectrum. First we address the case of
coverings of compact manifolds.
\begin{theorem} \label{intro1}
Given any compact manifold $M$, any noncompact 
covering manifold
$\tilde{M}$ of $M$, and any positive integer $G$, there is a metric on $M$ so that the lifted 
metric to $\tilde{M}$ has at least $G$ gaps in its $L^2$ essential spectrum.
\end{theorem}
It is noted that a regular Riemannian covering has the same volume growth as its deck transformation group. 
In particular this gives counterexamples to the Lu-Zhou conjecture. In general, however, our manifolds do not necessarily satisfy the sub-exponential volume growth. For instance, if the compact manifold 
$M$ is hyperbolic, then its fundamental group has exponential volume growth. 
This theorem also generalizes
work of O. Post \cite{post03} and of F. Lled\'o and O. Post \cite{lledpost08} who
used Floquet theory to make a similar construction for a special class of covering
manifolds. See also the more refined results of A. Khrabustovskyi \cite{khra12}
and limiting results obtained by P. Exner and O. Post \cite{exnerpost05}.

Our second main result removes the covering condition entirely and works for general
complete manifolds of bounded curvature and positive injectivity radius.
\begin{theorem} \label{intro2}
Let $(M,g_0)$ be a complete noncompact Riemannian manifold of bounded curvature and positive injectivity radius. Given any positive integer $G$ there is a metric
$g$ on $M$ such that $(M,g)$ has bounded curvature and positive injectivity radius,
the eigenvalues of $g$ with respect to $g_0$ are bounded above and below by positive 
constants, and the $L^2$ essential spectrum of $g$ has at least $G$ gaps. 
\end{theorem}
There are earlier papers which construct complete manifolds with gaps in their essential spectrum.
E. B. Davies and E. M. Harrell \cite{daviesharrell87} proved at least one gap for certain periodic conformally flat metrics. E. L. Green \cite{green94} showed that there are an arbitrary finite 
number of gaps for certain $2$-dimensional conformally flat metrics.

It is much more difficult to construct complete manifolds of bounded geometry with an infinite
number of gaps in the essential spectrum. J. Lott \cite{lott01} constructed a non-periodic, negatively curved, finite area $2$-dimensional surface with an infinite number of gaps in its essential spectrum.
Some intuition for the behavior of the spectrum of Riemannian manifolds comes from spectral results 
for Schr\"odinger operators on $\mathbb R^n$. 
For Schr\"odinger operators, $-\Delta+V(x)$, with a periodic potential $V$ on $\mathbb R^n$ the 
Bethe-Sommerfeld conjecture says that for $n\geq 2$ there can be at most a finite number of gaps in the essential spectrum. It has been solved for smooth potentials by L. Parnovski \cite{parnovski08}.
It would be interesting to understand whether this has a Riemannian manifold analogue.

Our theorems are, in fact, special cases of much more general results which hypothesize that
the manifold is made up of certain building blocks with sufficient control on
the geometry of the pieces. See Section 2 for the detailed definitions. Of course our manifolds
do not have nonnegative Ricci curvature, but some of them do have bounded positive scalar curvature.
In fact, our construction was motivated by the first author's \cite{schoen88} constructions of certain complete conformally flat metrics of constant positive scalar curvature.


Here is a sketch of our arguments. To prove the existence of gaps, there are two ingredients: first to show that certain real numbers are not in the spectrum and, second, to show that the intervals
between these numbers intersect non-trivially with the essential spectrum. Both steps reduce to various $L^2$ estimates. For the first step, we employ a space of approximate eigenfunctions 
coming from our building blocks. We decompose any smooth function as a sum of its projection on the approximate eigenspace plus the orthogonal part $u=u_0+u_1$. The projection $u_0$ can be controlled by the choice of our approximations. The orthogonal part $u_1$ over each building block also depends on that construction. Over the neck region, the $L^2$ norm of $u_1$ is controlled by the assumption of large Dirichlet eigenvalue.

For the second step, we construct functions which are sufficiently close to being a potential eigenfunction. A general method works for dimensions $n\geq 4$. For all dimensions including $n=2,3$, we require more control over our manifolds.   
 
The paper is organized as follows. In Section 2, we give the definitions of various classes of manifolds which all have bounded geometry and positive injectivity radius. We also obtain some preliminary
results that will be used later. Section 3 provides examples of constructions that satisfy the requirements of the classes defined in Section 2. Section 4 discusses the approximate kernel and corresponding estimates for $u_0$ and $u_1$. The estimate over the neck region is proven in Section 5.  Section 6 provides the proof of existence of points of the essential spectrum in the intervals. Finally, Section 7 completes the second step in our scheme and gives the proofs of main theorems.  
\section{\textbf{Definition of Classes of Manifolds and Preliminaries}}
We begin this section with a brief discussion of some of the basic geometric notions
which may be less familiar to analysts. The {\it injectivity radius} of a Riemannian
manifold $M$ at a point $p\in M$ is the largest number $r$ for which all points $q$
of distance $r$ from $p$ can be joined to $p$ by a unique minimizing geodesic. The
injectivity radius of $M$ is the infimum of the injectivity over all points of $M$. For a
non-compact manifold this number could be $0$ as is the case for a complete
non-compact manifold with finite volume. For compact manifolds with curvature bounded
from above and small injectivity radius, there is a closed geodesic whose length is
twice the injectivity radius. In Figure \ref{surface1}, the injectivity radius will be half the
circumference of the neck. 

Some of our results refer to {\it covering manifolds}, so we recall that a compact
manifold $M$ which is not simply connected has a universal covering manifold
$\tilde{M}$ so that $M$ is the quotient of $\tilde{M}$ by a group of deck transformations
which is isomorphic to the fundamental group of $M$. When $M$ has a Riemannian
metric $g$ we always endow $\tilde{M}$ with the metric which is pulled back from
$M$ under the covering projection. This makes $\tilde{M}$ into a complete Riemannian
manifold which is locally isometric to the base manifold. The deck transformations
are then global isometries of $\tilde{M}$. Basic examples of this
construction include flat metrics on the torus $\mathbb T^n$ in which the universal cover is
$\mathbb R^n$ and the deck group is $\mathbb Z^n$ generated by translations
in the directions of a lattice. Another basic example is a compact hyperbolic manifold
for which the universal cover is $\mathbb H^n$ and the manifold is the quotient 
of $\mathbb H^n$ by a discrete group of isometries acting freely on $\mathbb H^n$.

We now give a precise definition of the class of manifolds for which we
can show that there are an arbitrarily large number of spectral gaps. Each of these
will be a complete non-compact manifold with bounded curvature and positive injectivity
radius. Although some of them will arise as coverings of compact manifolds, there is
no a priori assumption of symmetry.
\subsection{Definitions}
We fix a finite collection of compact Riemannian $n$-manifolds $\mathcal X=\{X_1,\ldots,X_p\}$
and we denote the metric on $X_\alpha$ by $g_\alpha$. We let $\mathcal S$ denote
the union of the spectra of the $X_\alpha$. The spectral gaps for our manifold will
be roughly the first prescribed number of intervals of $\mathbb R\setminus\mathcal S$.

We first choose parameters $\bar{\rho}>0$, $\bar{\delta}>0$ and $\Lambda>1$. 
\begin{definition}
A complete manifold 
$(M,g)$ is said to be in the class $\mathcal M(\bar{\rho},\bar{\delta}, \Lambda)$ with respect to $\mathcal X$ if $M$ is the
union of a `core' domain $X$ and a `neck' region $N$ which overlap on a disjoint
union of annuli. We make the following assumptions on $X$ and $N$:
\begin{enumerate}
\item Each connected component $\hat{X}$ of $X$ is diffeomorphic to $X_\alpha$, for some
$\alpha\in\{1,\ldots,p\}$, with a finite union of balls removed. Under such an identification 
the domain $\hat{X}$ has boundary consisting of geodesic spheres in $(X_\alpha,g_\alpha)$ of radius at most $\bar{\rho}$.
\item Under the identification, the metrics $g$ and $g_\alpha$ are within $\bar{\delta}$
of each other in the $C^1(g_\alpha)$ norm. This means that the difference tensor $g-g_\alpha$
has $C^1$ norm (measured with respect to $g_\alpha$) bounded by $\bar{\delta}$.
\item The connected components of $X\cap N$ are of the form $B_{2\rho}\setminus B_\rho$
where $\partial B_\rho$ is a connected component of $\partial\hat{X}$ for some core
component $\hat{X}$ and $\rho\leq \bar{\rho}$.
\item The lowest Dirichlet eigenvalue of $N$ is at least $\Lambda$.
\end{enumerate}
\end{definition}
\begin{remark} If $N$ is unbounded in $M$, then the lowest
Dirichlet eigenvalue is defined to be the infimum of the lowest Dirichlet eigenvalue over
bounded subdomains of $N$.
\end{remark}
\begin{remark}
In examples, the assumption on the Dirichlet eigenvalue can often be obtained by scaling
the metric to make the neck small. 
\end{remark}
For part of our work we will need stronger assumptions in dimensions $2$ and $3$
which we now describe. 
\begin{definition}A manifold $M$ belongs to 
$\mathcal M_0(\bar{\rho},\bar{\delta},\Lambda)$ if it belongs to $\mathcal M(\bar{\rho},\bar{\delta},\Lambda)$ and there is an infinite subset $\mathcal S_0\subseteq \mathcal S$ such that
for any $\lambda\in\mathcal S_0$ there exists $X_\alpha\in\mathcal X$ such that $\lambda$
is an eigenvalue of $X_\alpha$ and there is an eigenfunction for $\lambda$ which
vanishes at the centers of the balls which are removed in the construction of $M$.
\end{definition}
Examples of manifolds of this type occur when one of the $X_\alpha$ has only a single
ball removed and the multiplicity of an infinite number of eigenvalues of $X_\alpha$
is at least two (so that an eigenfunction exists which vanishes at the center of the removed ball). For example if $X_\alpha$ is homogeneous there is an eigenfunction for any eigenvalue which vanishes
at a chosen point. If some $X_\alpha$ is a round sphere and the centers of the balls all lie
on an equator then there is an eigenfunction which vanishes at the centers for any
eigenvalue (for any degree there is a homogeneous harmonic polynomial which is
divisible by a chosen linear function).

A second class for which we will obtain our results in all dimensions is a class we
will call $\mathcal M_1(\bar{\rho},\bar{\delta},\Lambda)$. 
\begin{definition} A 
manifold is in $\mathcal M_1(\bar{\rho},\bar{\delta},\Lambda)$ if in addition to being in $\mathcal M(\bar{\rho},\bar{\delta},\Lambda)$,
each connected component $\hat{N}$ of $N$ has a diameter bounded by $\Lambda$, number
of boundary components bounded above by $\Lambda$, and volume bounded above by $\bar{\delta}$. 
\end{definition}
An example to think of is a Delaunay surface with small neck size, so that each component of $N$ 
is an annulus of small area. Recall that the Delaunay surfaces are the surfaces of revolution of
constant mean curvature in $\mathbb R^3$ (see \cite{eells87}). They look roughly
like a singly periodic tower of spheres connected by necks (see Figure \ref{surface1}). 

Note that our definition allows the possibility that the number of boundary
components of each connected component of $N$ be large but bounded.

\subsection{Preliminaries}
Here we collect various useful estimates on domains in a manifold with bounded geometry. 
First, it is observed that bounded geometry implies the metric is uniformly equivalent to the Euclidean metric \cite{eichhorn91} in balls of fixed radius. 

First, in this work we will need slight modifications of the standard Poincar\'e and Sobolev inequalities for functions in an annulus. 

We assume that we have an annulus $A=B_{r_0}\setminus B_{r_1}$ in $\mathbb R^n$
with a metric $g$ which is uniformly equivalent to the Euclidean metric; specifically for a positive
constant $C_1$ and all $a\in\mathbb R^n$
\[ C_1^{-1}\sum_{i=1}^na_i^2\leq\sum_{i,j=1}^ng^{ij}a_ia_j\leq C_1\sum_{i=1}^na_i^2.
\] 
Then the following estimates hold. 

\begin{lemma}\label{inequality} Suppose we have an annulus as above. 
\begin{enumerate}
\item \label{poincare} For any smooth function $f$ 
with $f=0$ on the inner boundary $\partial B_{r_1}$, there is a constant depending only on $r_0$, $r_1$ and $C_1$ such that,
\[ \int_Af^2\ dv\leq c\int_A\|\nabla_{g}f\|^2\ dv.
\]
\item \label{sobolev} Assume $n\geq 3$. For any smooth function $f$ on $A$ with $f=0$ on $\partial B_{r_0}$,
there is a constant depending only on $n$ and $C_1$ (independent of $r_0$ and $r_1$) such that, 
\[(\int_A f^{\frac{2n}{n-2}}\ dv)^{\frac{n-2}{n}} \leq c\int_A\|\nabla_g f\|^2\ dv.
\]
\end{enumerate}
\end{lemma}

\begin{proof} 
For the  Poincar\'e inequality, it is noted that for the Euclidean case the constant is the inverse of the lowest eigenvalue
for the problem with Dirichlet condition on the inner boundary and Neumann on the outer
boundary. Because the metric $g$ is uniformly equivalent to the Euclidean metric, each term of the inequality only varies within multiplicative bounds determined by that equivalence, so the result follows.

The Sobolev inequality follows in a standard way from the corresponding $L^1$
inequality
\[ (\int_Af^{\frac{n}{n-1}}\ dv)^{\frac{n-1}{n}}\leq c\int_A\|\nabla_g f\|\ dv
\]
for functions $f$ which vanish on the outer boundary.
That, in turn, is equivalent to the isoperimetric inequalty,
\[ Vol(\Omega)\leq cVol(\partial\Omega\setminus\partial B_{r_1})^{\frac{n}{n-1}}
\]
for any $\Omega\subseteq A$. 
Note that it suffices to prove the inequality for the
Euclidean metric since both sides have bounded ratio (with bound depending on $C_1$)
with the corresponding quantity for the metric $g$. We note that the standard isoperimetric 
inequality for $\Omega$ may be written
\[ Vol(\Omega)\leq c\Big(Vol(\partial\Omega\setminus\partial B_{r_1})+Vol(\partial\Omega\cap
\partial B_{r_1})\Big)^{\frac{n}{n-1}}.
\]  
Next we observe that the radial projection map $P:A\to\partial B_{r_1}$ given by
$P(x)=r_1 x/|x|$ reduces volumes of hypersurfaces. It thus follows that
\[ Vol(P(\partial\Omega\setminus\partial B_{r_1}))\leq Vol(\partial\Omega\setminus \partial B_{r_1}).
\]
On the other hand any ray through a point of $\partial\Omega\cap\partial B_{r_1}$
must intersect $\partial\Omega$ at a second point, and so we have 
\[ \partial\Omega\cap\partial B_{r_1}\subseteq P(\partial\Omega\setminus\partial B_{r_1}).
\]
Combining this information with the isoperimetric inequality we have
\[ Vol(\Omega)\leq 2^{\frac{n}{n-1}}cVol(\partial\Omega\setminus\partial B_{r_1})^{\frac{n}{n-1}}.
\]
This completes the proof of the desired isoperimetric inequality and assertion \ref{sobolev}
follows as indicated above.
\end{proof}

We also need the following version of a logarithmic cut-off function argument.
\begin{lemma} \label{logcutoff} Suppose $B_{r}$, a ball in $\mathbb R^n$, is equipped with a metric equivalent to the Euclidean metric. For any $\epsilon$, there are small $\rho$ and a smooth function $\zeta$, which is $1$ for $x$ from $\partial B_{\sqrt{\rho}}$ and $0$ for $x$ in $\partial B_{\rho}$, such that the following holds. For any smooth function $u_1$,  
\begin{align*}
\int_{B_{\sqrt{\rho}}\setminus B_{\rho}}\|\nabla\zeta\|^2u_1^2\ dv &\leq c\epsilon\int_{B_r}(u_1^2+\|\nabla u_1\|^2)\ dv;\\
\int_{B_{\sqrt{\rho}}\setminus B_{\rho}}u_1^2\ dv &\leq c\epsilon\int_{B_r}(u_1^2+\|\nabla u_1\|^2)\ dv.
\end{align*}
Here $c$ is a constant depending on $r$ and bounds on the eigenvalues of the metric with
respect to the euclidean metric. 

\end{lemma} 
\begin{proof}
Let $A=B_{\sqrt{\rho}}\setminus B_{\rho}$ and we set
\[ \zeta(r)=\frac{\log(r/\rho)}{\log(1/\sqrt{\rho})}\ \mbox{for}\ \rho\leq r\leq \sqrt{\rho}.
\]
Note that the function $\zeta$ we have chosen is not smooth but only Lipschitz continuous.
It is a standard argument to see that such a $\zeta$ can be approximated by
smooth functions in the $W^{1,2}$ norm so that we can justify this choice. Also, it suffices to prove the first inequality because, for our choice of $\zeta$,
$1<|\nabla \zeta|$ on $A$.\\

We first consider $n\geq 3$. In this
case we use the H\"older inequality to obtain
\[ \int_{A}\|\nabla\zeta\|^2u_1^2\ dv\leq (\int_{A}
\|\nabla\zeta\|^n\ dv)^{2/n}(\int_{A}u_1^{2n/(n-2)}\ dv)^{(n-2)/n}.
\]
From the definition of $\zeta$ and the conditions on the metric on the annulus we have
\[ \int_{A}\|\nabla\zeta\|^n\ dv\leq c|\log(\rho)|^{-n}\int_{\rho}^{\sqrt{\rho}}
r^{-1}dr\leq c|\log(\rho)|^{1-n}.
\]
Thus for any $\epsilon>0$, when $\rho$ is small enough we have
\[ \int_{B_{\sqrt{\rho}}\setminus B_\rho}\|\nabla\zeta\|^2u_1^2\ dv\leq \epsilon
(\int_{B_{\sqrt{\rho}}\setminus B_\rho}u_1^{2n/(n-2)}\ dv)^{(n-2)/n}.
\]
Now if $\psi$ is a cut-off function, which is $1$ on $B_\rho$ and supported in $B_{r}$, then
we have
\[ (\int_{B_{\sqrt{\rho}}\setminus B_\rho}u_1^{2n/(n-2)}\ dv)^{(n-2)/n}\leq (\int_{B_{r}}(\psi u_1)^{2n/(n-2)}\ dv)^{(n-2)/n}\leq c\int_{B_{r}}\|\nabla (\psi u_1)\|^2\ dv.
\]
Here we have used the Sobolev inequality, Lemma \ref{inequality}(\ref{sobolev}), for functions vanishing on the outer boundary of the annulus $B_{r}\setminus B_\rho$. Since the gradient of $\psi$ is bounded
we obtain
\[ (\int_{B_{\sqrt{\rho}}\setminus B_\rho}u_1^{2n/(n-2)}\ dv)^{(n-2)/n}\leq (\int_{B_{r}}(\psi u_1)^{2n/(n-2)}\ dv)^{(n-2)/n}\leq c\int_{B_r}(u_1^2+\|\nabla u_1\|^2)\ dv.
\]
Combining with our previous inequality we obtain,
\[ \int_{B_{\sqrt{\rho}}\setminus B_{\rho}}\|\nabla\zeta\|^2u_1^2\ dv\leq c\epsilon\int_{B_r}(u_1^2+\|\nabla u_1\|^2)\ dv.
\]

For $n=2$ we can obtain the conclusion in a slightly different way. We make the same
choice of $\zeta$, and we come to the problem of estimating the term
\[  \int_{B_{\sqrt{\rho}}\setminus B_\rho}\|\nabla\zeta\|^2u_1^2\ dv= c|\log(\rho)|^{-2}
\int_{B_{\sqrt{\rho}}\setminus B_\rho}r^{-2}u_1^2\ dv.
\]
We observe that since the metric is near Euclidean in an appropriate annulus $B_{r}\setminus
B_\rho$ where $r$ is a fixed radius, we may do the estimate in the Euclidean metric.
In this case, the volume form $|x|^{-2}dx^1dx^2$ is that of the cylinder $\mathbb R\times \mathbb S^1$
with coordinates $t=\log(|x|)$ and the polar coordinate $\theta$. The annulus now becomes the
cylinder $[\log(\rho),1/2\log(\rho)]\times \mathbb S^1$.
 
Consider the eigenvalue problem with boundary conditions which are Dirichlet at $t=\log(r)$
and Neumann at $t=\log(\rho)$. As the metric is rotationally symmetric, we can use separation of variables to see that the first eigenvalue is inversely proportional to the length of this tube. Choosing $\psi$ to be a cutoff function of $t$ which is $0$ for $t\geq\log(r)$ and $1$ for $t\leq\log(r)-1$, we may
apply the Poincar\'e inequality to obtain,
\[ \int_{A}r^{-2}u_1^2\ dv\leq \int_{B_{r}\setminus B_\rho}
(\psi u_1)^2r^{-2}\ dv\leq c|\log(\rho)|\int_{B_{r}}\|\nabla (\psi u_1)\|^2\ dv.
\]
Here we also use the fact that the Dirichlet integral is conformally invariant. We have chosen $\psi$ so that it has 
bounded derivatives, so we obtain as in the case $n\geq 3$,
\[ \int_{B_{\sqrt{\rho}}\setminus B_{\rho}}\|\nabla\zeta\|^2u_1^2\ dv\leq c|\log(\rho)|^{-1}\int_{B_r}(u_1^2+\|\nabla u_1\|^2)\ dv.
\]

\end{proof}

\section{\textbf {Examples}}
We begin by considering coverings of complete manifolds. Our first result gives conditions
under which a covering $\tilde{M}$ of a complete manifold in class $\mathcal M(\bar{\rho},\bar{\delta},\Lambda)$ also lies in the same class.

\begin{proposition}\label{coverings} Suppose $M\in \mathcal M(\bar{\rho},\bar{\delta},\Lambda)$ is a complete manifold. Let $\tilde{M}$ be a Riemannian covering of $M$ with the property that the 
covering projection $\Pi$ is a diffeomorphism from each connected component of $\Pi^{-1}(X)$ to its image. We then have $\tilde{M}\in \mathcal M(\bar{\rho},\bar{\delta},\Lambda)$.
\end{proposition}

\begin{remark} The condition on the covering $\tilde{M}$ is automatic if $X$ is simply connected.
In general if the covering $\tilde{M}$ corresponds to a subgroup $\Gamma$ of $\pi_1(M)$, the
condition follows from the condition that $\pi_1(\hat{X})\subseteq\Gamma$ for each component
$\hat{X}$ of $X$ where the base point is understood to lie in $\hat{X}$.
\end{remark}

\begin{proof} To show that $\tilde{M}\in \mathcal M(\bar{\rho},\bar{\delta},\Lambda)$ we let
$\tilde{X}=\Pi^{-1}(X)$ and observe that the condition that $\Pi$ is a diffeomorphism (hence
an isometry) on each component of $\tilde{X}$ implies the conditions (1) and (2) for
$\tilde{X}$. We let $\tilde{N}=\Pi^{-1}(N)$, and we will show that 
$\lambda_1(\tilde{N})\geq\lambda_1(N)$.
For each component $\hat{N}$ of $N$ we have $\lambda_1(\hat{N})\geq\lambda_1(N)$
since $\lambda_1(N)$ is the infimum of $\lambda_1$ over its components. By 
\cite{FCS80} it follows that this condition is equivalent to the existence of a positive solution $u$
on $\hat{N}$
of the differential inequality $\Delta u+\lambda_1(\hat{N})u\leq 0$. Lifting $u$ to any
component of $\Pi^{-1}(\hat{N})$ shows that each such component has first Dirichlet eigenvalue
at least $\lambda_1(N)$ as required. We have thus verified all four of the properties and
have shown that $\tilde{M}\in \mathcal M(\bar{\rho},\bar{\delta},\Lambda)$.
\end{proof}

\begin{remark}\label{coveringsm0}
It is clear that a similar statement also holds for $\mathcal M_0(\bar{\rho},\bar{\delta},\Lambda)$ but not $\mathcal M_1(\bar{\rho},\bar{\delta},\Lambda)$. The problem is that a connected component of the pre-image of a component of $N$ might cover many (or infinite) times. The statement holds then if there is a finite uniform bound for the number of pre-images of a point in any component of $\Pi^{-1}(N)$.  
\end{remark}
 
Now we show how to construct a metric on a non-compact covering which lies in our class. Our model will be a standard unit sphere. Recall that the metric on the unit $n$-sphere can be written as,
\[\frac{4\epsilon^2}{(\epsilon^2+ r^2)^2}(dr^2+r^2 d_{\mathbb S^{n-1}}),\]
where $dr^2+r^2 d_{\mathbb S^{n-1}}$ is just the Euclidean metric. 
\begin{figure}[h]
\includegraphics[width=12cm]{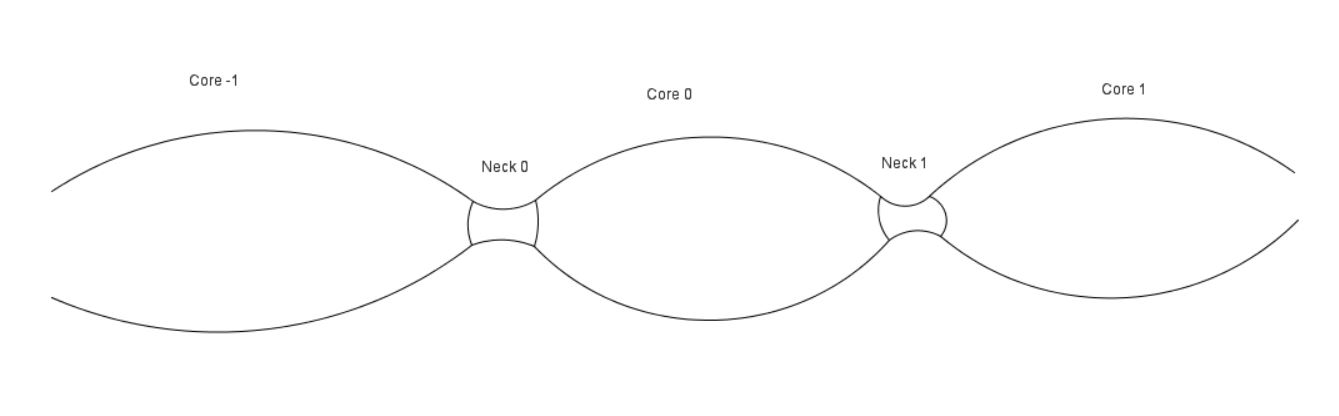}
\caption{Surface with small necks}
\label{surface1}
\end{figure}
\begin{proposition}\label{covering_gaps} We fix some arbitrarily small $\bar{\rho}$ and arbitrarily large $\Lambda$. Given any compact manifold $M$, any noncompact 
covering manifold
$\tilde{M}$ of $M$, there is a metric $g$ on $M$ so that $(\tilde{M},\tilde{g})\in \mathcal M(\bar{\rho},0,\Lambda)$. In fact, we also have $(\tilde{M},\tilde{g})\in  \mathcal M_0(\bar{\rho},0,\Lambda)$.
\end{proposition}

\begin{proof} We begin with a metric $g_0$ which contains an isometric
 copy of the unit ball in $\mathbb R^n$. Thus we can find local coordinates so that for
 $|x|\leq 1$ the metric $g_0$ is the euclidean metric. We now define $g$ to be a metric
 conformal to $g_0$ of the form $g=u^2g_0$ where we take $u$ to be a smooth approximation
 to the function which is, 
 $\begin{cases} &=2(\epsilon+\epsilon^{-1}|x|^2)^{-1} \mbox{ for } |x|\leq 1,\\
 &=2(\epsilon+\epsilon^{-1})^{-1} \mbox{ otherwise. }
 \end{cases}$ 
 
 Here $\epsilon>0$
 is a number which we will choose small. The metric we have chosen is isometric to a large 
 portion of the unit sphere $\mathbb S^n$; in fact, the radius in $\mathbb S^n$ of a Euclidean sphere
 $|x|=r\leq 1$ is given by $\sin^{-1}(ru)=\sin^{-1}(2r(\epsilon+\epsilon^{-1}r^2)^{-1})$. 
 
Now we set, for $|x|$ the Euclidean distance,
\begin{itemize}
\item $\rho=\sin^{-1}(2(\epsilon+\epsilon^{-1})^{-1})$ (for $\epsilon$ small, $\rho\approx 2\epsilon$),
\item $r_0$ is a number such that the sphere $|x|=r_0$ has radius $2\rho$ in $\mathbb S^n$ ($r_0\approx 1/2$),
\item $X=\{|x|\leq r_0\}$,
\item $N=M\setminus \{|x|\leq r_0\}$.
\end{itemize}

Let $\lambda=\lambda_1(N,g_0)>0$. On $N$ we observe that $c_1\epsilon\leq u\leq c_2\epsilon$ where $c_1$ and $c_2$
 are positive constants. Therefore, for any function $f$ vanishing on $\partial N$
 we have
 \[ \int_Nf^2\ dv_g\leq (c_2\epsilon)^n\int_Nf^2\ dv_{g_0}\leq (c_2\epsilon)^n\lambda^{-1}
 \int_N\|\nabla_{g_0}f\|^2\ dv_{g_0}\leq \frac{c_2^n}{c_1^{n-2}}\epsilon^2\lambda^{-1}
 \int_N\|\nabla_gf\|^2\ dv_g.
 \]
 It follows that $\lambda_1(N,g)\geq c\lambda\epsilon^{-2}$ for some positive constant $c$.
 Thus by choosing $\epsilon$ small we have $(M,g)\in\mathcal M(\bar{\rho},0,\Lambda)$
 for $\bar{\rho}$ as small as we wish and $\Lambda$ as large as we wish where our model
 manifold is the standard unit $n$-sphere. By Proposition \ref{coverings} the lifted
 metric on $\tilde{M}$ is also in class $\mathcal M(\bar{\rho},0,\Lambda)$. The last statement follows because our model is homogeneous and, thus, Remark \ref{coveringsm0} applies.  
\end{proof}

We now make a similar construction for any complete manifold with bounded geometry;
that is, bounded curvature and positive injectivity radius. That is we remove the condition
that our complete manifold cover a compact manifold.

\begin{proposition} \label{complete}
We fix some arbitrarily small $\bar{\rho}$ and arbitrarily large $\Lambda$.
Let $(M,g_0)$ be a complete noncompact Riemannian manifold of bounded curvature and positive injectivity radius. There is a metric $g$ on $M$ such that $(M,g)$ has bounded curvature and positive injectivity radius,
the eigenvalues of $g$ with respect to $g_0$ are bounded above and below by positive 
constants, and $({M},{g})\in \mathcal M(\bar{\rho},0,\Lambda)$. In fact, we also have $({M},{g})\in  \mathcal M_0(\bar{\rho},0,\Lambda)$.
\end{proposition}
\begin{figure}[h]
\includegraphics[width=12cm]{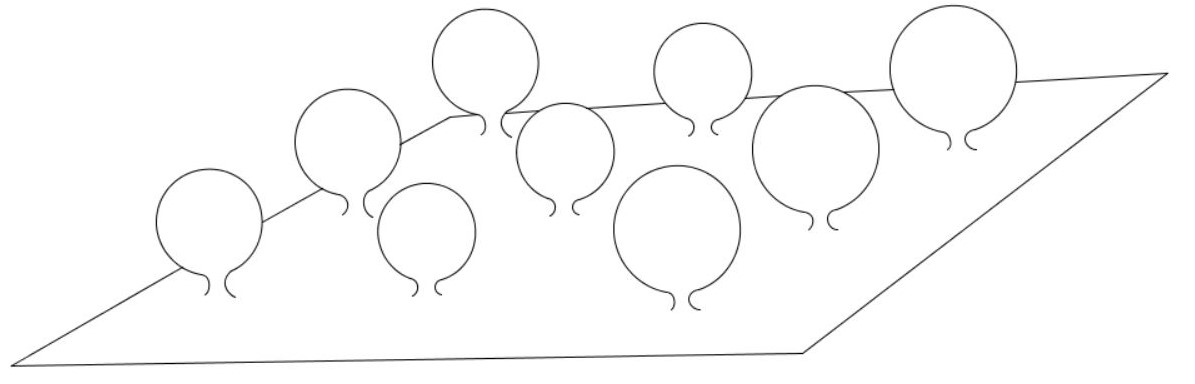}
\caption{Core consists of Sphere-like Bubbles}
\label{completefigure1}
\end{figure}
\begin{proof} Since we are allowed to rescale $g_0$ by a constant, there is no loss of generality
in assuming that the injectivity radius of $g_0$ is at least $2$. We now choose a maximal disjoint
family of balls $B_1(p_j)$, $j=1,2,\ldots$ of radius $1$. Next we deform the metric $g_0$ to
a new metric $g_1$ which has bounded curvature and positive injectivity radius such that the
ball of radius $1$ about each point $p_j$ is Euclidean. To do this we choose normal coordinates
centered at $p_j$ and write $g_0$ in the form
\[ g_0=dr^2+r^2g(r),\ 0\leq r\leq 2
\]
where $g(r)$ is a smooth family of metrics on $\mathbb S^{n-1}$ of bounded curvature with $g(0)$
equal to the standard unit metric. We choose a smooth non-decreasing function $\zeta(r)$
such that
$$\begin{cases}\zeta(r)&=0 \text{ for } 0\leq r\leq 1,\\
 \zeta(r)&=r \text{ for } 3/4\leq r\leq 2\end{cases}.$$
 
  We then define the
metric $g_1$ in $B_2(p_j)$ by setting
\[ g_1=dr^2+r^2g_{\zeta(r)},\ 0\leq r\leq 2.
\]
Since $\zeta$ has bounded second derivatives and each metric $g_r$ has bounded curvature
it follows that the curvature of $g_1$ is bounded. By construction the metric $g_1$ is euclidean
in $B_1(p_j)$ for each $j$. Since the metric $g_1$ is uniformly equivalent to $g_0$, the local
volumes of small balls are bounded below by those of corresponding Euclidean balls and it
follows that the injectivity radius of $g_1$ is bounded from below. 

We now construct $g$ by deforming $g_1$ in $B_1(p_j)$ as in the proof of Proposition \ref{covering_gaps}. Our model manifold is again the unit $\mathbb S^n$, and we take,
\begin{itemize}
\item $\rho=\sin^{-1}(2(\epsilon+\epsilon^{-1})^{-1})$ (for $\epsilon$ small, $\rho\approx 2\epsilon$),
\item $r_0$ is a number such that the sphere $|x|=r_0$ has radius $2\rho$ in $\mathbb S^n$ ($r_0\approx 1/2$),
\item $X$ to be the union of the $B_1(p_j)$,
\item $N=M\setminus \cup_{p_j} B_{r_0}(p_j,g_1)$.
\end{itemize}

In order to show that $(M,g)\in\mathcal M(\bar{\rho},0,\Lambda)$ we must show that the
lowest Dirichlet eigenvalue of $N$ with respect to $g_1$ is positive. If we show this then the
same argument as above shows that $\lambda_1(N,g)\geq c\lambda\epsilon^{-2}$. The idea is that we
have removed enough balls from $M$.

To be precise, for smooth function f in $A_j=B_2(p_j,g_1)\setminus B_{r_0}(p_j,g_1)$ and $f=0$ on the inner boundary, by Lemma \ref{inequality}(\ref{poincare}), we have, 
\[ \int_{A_j}f^2\ dv_1\leq c\int_{A_j}\|\nabla_{g_1}f\|^2\ dv_1.
\]

Since $r_0\approx 1/2$, the constant only depends on the equivalence between $g_1$ and the Euclidean metric.  

It now follows that if we take any smooth function $f$ with bounded support on $N$ which is
zero on $\partial N$, we can extend $f$ to all of $M$ by setting it to $0$ on $M\setminus N$,
and we have from above
\[ \int_{B_2(p_j,g_1)}f^2\ dv_1\leq c\int_{B_2(p_j,g_1)}\|\nabla_{g_1}f\|^2\ dv_1.
\]
Note that the balls $B_2(p_j,g_1)=B_2(p_j,g_0)$ cover $M$ since the collection $B_1(p_j)$
was chosen to be a maximal disjoint collection of unit balls. (A point $q$ of distance more than
$2$ from all of the $p_j$ would have the property that $B_1(q)$ is disjoint from all of the
$B_1(p_j)$). Therefore we have
\[ \int_Nf^2\ dv_1\leq \sum_j\int_{B_2(p_j,g_1)}f^2\ dv_1\leq c\sum_j\int_{B_2(p_j,g_1)}\|\nabla_{g_1}f\|^2\ dv_1.
\]
Now for any point $p$ we let $k(p)$ denote the number of balls $B_2(p_j)$ to which $p$
belongs. We may write the term on the right
\[ \sum_j\int_{B_2(p_j,g_1)}\|\nabla_{g_1}f\|^2\ dv_1=\int_M k(p)\|\nabla_{g_1}f\|^2\ dv_1.
\]
We claim that the function $k(p)$ is uniformly bounded. In fact, if for some $j$, $p\in B_2(p_j)$, then $B_3(p)$ must contain $B_1(p_j)$. Since the unit
balls $B_1(p_j)$ are disjoint there can only be a bounded number of such balls by volume
considerations $Vol(B_3(p)\leq cVol(B_1(p_j)$ for each $j$. (Note that metrics $g_0$
and hence $g_1$ are uniformly bounded in terms of the euclidean metric on $B_2(p)$ for any
$p$.) Therefore it follows that 
\[  \int_Nf^2\ dv_1\leq \sum_j\int_{B_2(p_j,g_1)}f^2\ dv_1\leq c\int_N\|\nabla_{g_1}f\|^2\ dv_1
\]
for any smooth function on $N$ with bounded support vanishing on $\partial N$. This shows
that $\lambda_1(N)>0$ and completes the proof that $({M},{g})\in \mathcal M(\bar{\rho},0,\Lambda)$. For the last statement, we apply Remark \ref{coveringsm0} again.

\end{proof}

\section{\textbf{The Approximate Eigenspace}}

We fix a number $\lambda\in\mathbb R\setminus\mathcal S$ and let $d\leq \text{dist}(\lambda,\mathcal S)$. In this section we construct a closed subspace $E_0$ of $L^2(M)$ consisting of smooth
functions whose restriction to a connected component $\hat{X}_\alpha$ of $X$ is an
approximate union of the Neumann eigenspaces for eigenvalues less than $\lambda$.
We will need the following properties of $E_0$.
\begin{prop}\label{approx_espace} Assume that $\bar{\rho}$ and $\bar{\delta}$ are chosen sufficiently
small and $M\in\mathcal M(\bar{\rho},\bar{\delta},\Lambda)$. For $u_0\in E_0$ with compact support we have
\[ \int_Mu_0^2\leq c\int_M(\lambda u_0^2-\|\nabla u_0\|^2)\ dv.
\]

For a smooth function $u_1\in E_0^\perp$ with compact support we have
\[ \int_Xu_1^2\leq c\int_X(\|\nabla u_1\|^2-\lambda u_1^2)\ dv.
\]

Each constant here only depends on $d$ and the geometry of $\mathcal X$. In particular, it depends on the equivalence of each metric $g_\alpha$ with the Euclidean metric and estimates on eigenfunctions on $X_\alpha$ with eigenvalues less than $\lambda$. Consequently, the estimates hold for $\lambda$ in a compact interval disjoint from $\mathcal S$.  \\
\end{prop}
\textbf{Construction:} The space $E_0$ is defined to be the direct sum of finite dimensional spaces $\hat{E}$
of functions supported in a connected component $\hat{X}_\alpha$ of $X$. We let
$E_\alpha$ be the direct sum of the eigenspaces of $X_\alpha$ with eigenvalue less than
$\lambda$. Assume that a boundary component of $\hat{X}_\alpha$ is a sphere of radius
$\rho$ with $\rho\leq\bar{\rho}$. We then let $\zeta$ denote a cutoff function of the geodesic 
distance $r$ on
$\hat{X}_\alpha$ which is $1$ for $r\geq \sqrt{\rho}$ and $0$ for $r\leq \rho$. We then
define $E_0=\{\zeta v:\ v\in E_\alpha\}$. \\

With a careful choice of $\zeta$ we can now
prove the inequalities of Proposition \ref{approx_espace}. 
Let $\xi=(dv)(dv_\alpha)^{-1}$ be the
ratio of the volume forms, and we note that $\xi$ is near one and its derivative is small for small 
$\bar{\delta}$.

Also, since $X_\alpha$ is a compact manifold we have for $v\in E_\alpha$
\[ \sup\{v(x)^2+\|\nabla_\alpha v(x)\|^2:\ x\in X_\alpha\}\leq c\int_{X_\alpha}v^2\ dv_\alpha,
\]
where the constant $c$ depends on $X_\alpha$ and $\lambda$. 
\begin{proof} (\textbf{Prop \ref{approx_espace}})
By the definition of
$E_\alpha$ we have for $v\in E_\alpha$
\[ \int_{X_\alpha}\|\nabla_\alpha v\|^2\ dv_\alpha\leq \lambda_\alpha\int_{X_\alpha}v^2\ dv_\alpha
\]
where $\lambda_\alpha$ is the largest eigenvalue of $X_\alpha$ which is less than $\lambda$.
It follows that
\[ \int_{X_\alpha}v^2\ dv_\alpha\leq c\int_{X_\alpha}(\lambda v^2-\|\nabla_\alpha v\|^2)\ dv_\alpha
\]
where $c=d^{-1}$. Now using the assumption that the metrics
$g_\alpha$ and $g$ are close on $\hat{X}_\alpha$ we have for $u_0=\zeta v$,
\[ \int_{\hat{X}_\alpha}u_0^2\ dv\leq c\int_{\hat{X}_\alpha\setminus B_{\sqrt{\rho}}}(\lambda u_0^2-\|\nabla u_0\|^2)\ dv+c\lambda\int_{\hat{X}_\alpha\cap B_{\sqrt{\rho}}}v^2\ dv_\alpha.
\]
Because of the supremum estimate of $v$, the second term on the right is a small constant
times the term on the left provided $\rho$ is chosen small. Thus we can absorb it back to
the left and remove it from the inequality. We then have
\[ \int_{\hat{X}_\alpha}u_0^2\ dv\leq c\int_{\hat{X}_\alpha}(\lambda u_0^2-\|\nabla_\alpha u_0\|^2)\ dv
+c\int_{B_{\sqrt{\rho}}\setminus B_\rho}\|\nabla (\zeta v)\|^2\ dv.
\]
The second term on the right is bounded by a constant times
\[ \int_{B_{\sqrt{\rho}}\setminus B_\rho}(\|\nabla \zeta\|^2 v^2+\zeta^2\|\nabla v\|^2)\ dv.
\]
By the supremum estimate on $\|\nabla v\|$ and the smallness of the annulus, the second term can be absorbed into the
left. Again from the supremum estimate on $v$ the first term is bounded by
\[ c(\int_{B_{\sqrt{\rho}}\setminus B_\rho}\|\nabla\zeta\|^2\ dv)(\int_{\hat{X}_\alpha}u_0^2\ dv).
\] 
Now by a computation similar (and easier) to Lemma \ref{logcutoff}, the Dirichlet integral of $\zeta$ is small on the annulus if $\rho$ is small.
Thus it can be absorbed to the left and
we have proven the first inequality of Prop. \ref{approx_espace}.\\

To prove the second inequality we let $u_1\in E_0^\perp$ and again we may focus
on a single connected component $\hat{X}_\alpha$. We note first that the condition
that $u_1\in E_0^\perp$ is equivalent to the statement that $\zeta u_1$ is in $E_\alpha^\perp$
with respect to the volume defined by $g$. Consequently, $\zeta\xi u_1$ is orthogonal to $E_\alpha$ with respect to the metric $g_\alpha$.
By the variational characterization of the eigenvalues below $\lambda$ we then have
\[ \int_{\hat{X}_\alpha}(\zeta\xi u_1)^2\ dv_\alpha\leq c\int_{X_\alpha}(\|\nabla_\alpha (\zeta\xi u_1)\|^2
-\lambda(\zeta\xi u_1)^2)\ dv_\alpha.
\]
Here $c=d^{-1}$ again. Using the assumption that $\xi$ is close to $1$ in $C^1$-norm on the support of $\zeta$ and
that the metrics $g_\alpha$ and $g$ are close we readily obtain,
\[ \int_{\hat{X}_\alpha}(\zeta u_1)^2\ dv\leq c\int_{\hat{X}_\alpha}(\|\nabla u_1\|^2-\lambda u_1^2)\ dv
+c\int_{B_{\sqrt{\rho}}\setminus B_\rho}(\|\nabla \zeta\|^2+1)u_1^2\ dv.
\]
Then it follows that,
\[ \int_{\hat{X}_\alpha}u_1^2\ dv\leq c\int_{\hat{X}_\alpha}(\|\nabla u_1\|^2-\lambda u_1^2)\ dv
+c\int_{B_{\sqrt{\rho}}\setminus B_\rho}(\|\nabla \zeta\|^2+2)u_1^2\ dv.
\]
The second term on the right is controlled using Lemma \ref{logcutoff} for $r$ being the mimimum of injectivity radii on our models. Thus, 

\[ \int_{\hat{X}_\alpha}u_1^2\ dv\leq c\int_{\hat{X}_\alpha}(\|\nabla u_1\|^2-\lambda u_1^2)\ dv
+c\epsilon\int_{\hat{X}_\alpha}(u_1^2+\|\nabla u_1\|^2)\ dv.
\]
This clearly implies 
\[ \int_{\hat{X}_\alpha}u_1^2\ dv\leq c\int_{\hat{X}_\alpha}(\|\nabla u_1\|^2-\lambda u_1^2)\ dv.
\]
Summing these inequalities over the components of $X$ then completes the proof of the
second inequality of Proposition \ref{approx_espace}.

\end{proof}
\section{\textbf{Contribution on the Neck Region}}
We now consider the neck region $N$ and prove the following estimate.
\begin{proposition} \label{neckestimate} For any smooth function
$u$ with compact support on $M$ we have the bound,
\begin{equation} \label{neck} \int_Nu^2\ dv\leq \epsilon\int_{M}(u^2+\|\nabla u\|^2)\ dv
\end{equation}
where $\epsilon$ can be made arbitrarily small by choosing $\bar{\rho}, \bar{\delta}$ small
and $\lambda_1(N)$ large.
\end{proposition}
\begin{proof}
Recall that we are assuming that the first Dirichlet eigenvalue of $\lambda_1(N)$ is large. This implies that we have the Poincar\'e inequality
\[ \int_Nv^2\ dv\leq \lambda_1(N)^{-1}\int_N\|\nabla v\|^2\ dv
\]
for any smooth function $v$ with compact support and with $v=0$ on $\partial N$. 

We apply the Poincar\'e inequality with $v=\zeta u$ where $\zeta$ is a function
which is $1$ on $N\setminus X$ and cuts off to $0$ on each of the annuli $B_{2\rho}\setminus
B_\rho$; those annuli are the components of $X\cap N$. We thus obtain,
\[ \int_{N\setminus X}u^2\ dv\leq 2\lambda_1(N)^{-1}\int_N(\zeta^2\|\nabla u\|^2+\|\nabla\zeta\|^2u^2)\ dv.
\]
The second term on the right can be controlled by using Lemma \ref{logcutoff} again. Here each annuli is of the form $B_{2\rho}\setminus B_\rho$; by a translation, each could be written as $B_{\sqrt{\rho_1}}\setminus B_{\rho_1}$ for some $\rho_1 \approx\rho^2$ and the proof carries over. 

So \[ \int_N\|\nabla \zeta\|^2u^2\ dv\leq c \epsilon \int_M(u^2+\|\nabla u\|^2)\ dv. 
\]
Combining this with the previous inequality completes the proof. 
\end{proof}
\begin{remark}
Here, thanks to factor $\lambda_1(N)^{-1}$ we actually only need a weaker version of Lemma \ref{logcutoff} where we could replace $c\epsilon$ by $c$. For that purpose, when $n\geq 3$ we can choose a standard cut-off function and apply the H\"older and Sobolev inequality.   
\end{remark}
\begin{remark}
The assumption that $\bar{\delta}$ be small is required only because both the arguments for $n\geq 3$ and for $n=2$ use comparison with the Euclidean metric. The bound on $\lambda$ depends
essentially on the largeness of $\Lambda_1(N)$. 
\end{remark}
\section{\textbf{Non-membership of Essential Spectrum}}
The main theorem of this section says that if we choose any number $\lambda$ which is not in the
set $\mathcal S$, the union of the spectra of the $X_\alpha$, then it will not
be in the spectrum of $M$ provided that $\bar{\rho}$ and $\bar{\delta}$ are chosen small and
$\Lambda$ is large enough. The precise statement is the following.
\begin{theorem}\label{first_main} Let $\lambda$ be any real number which is not in $\mathcal S$ and let $d\leq \text{dist}(\lambda, \mathcal S)$. If $\bar{\rho}$ and $\bar{\delta}$ are chosen small enough (depending on $d$ and the geometry of $\mathcal X$) then for any manifold $M\in\mathcal M(\bar{\rho},\bar{\delta},\Lambda)$, the number $\lambda$ is
not in the spectrum of $M$. Also, as explained in Prop. \ref{approx_espace}, the statement holds for any $\lambda$ in a compact interval disjoint from $\mathcal S$.
\end{theorem} 

\begin{proof}
Consider any function $u$ which is smooth and of compact support on $M$. We now
decompose $u=u_0+u_1$ where $u_0\in E_0$ and $u_1\in E_0^\perp$. Since $E_0$ is a disjoint union of finite dimensional vector spaces of smooth functions, $u_0$ is smooth and so is $u_1=u-u_0$. Also $u_0$ has support on $X$ and, thus $u_1=u$ on $N\setminus X$. \\ 

We then have, by Prop. \ref{approx_espace},  
\begin{align} 
\label{u0} \int_Mu_0^2\ dv &\leq c\int_M(\lambda u_0^2-\|\nabla u_0\|^2)\ dv,\\
\label{u1raw}
\int_Xu_1^2\ dv &\leq c\int_X(\|u_1\|^2-\lambda u_1^2)\ dv.
\end{align}
To get a bound on the $L^2$ norm of $u_1$
on all of $M$ we use (\ref{neck}) as follows,
\[ \int_Mu_1^2\ dv=\int_Xu_1^2\ dv +\int_{N\setminus X} u_1^2\ dv\leq c\int_X(\|u_1\|^2-\lambda u_1^2)\ dv
+\epsilon\int_M(u^2+\|\nabla u\|^2)\ dv,
\]
where we have used the fact that $u_1=u$ in $N\setminus X$. Now we have
\[ \int_X(\|\nabla u_1\|^2-\lambda u_1^2)\ dv\leq \int_M(\|\nabla u_1\|^2-\lambda u_1^2)\ dv+\lambda
\int_{N\setminus X} u^2\ dv.
\]
We can then apply the neck estimate a second time to obtain,
\begin{equation} \label{u1}  \int_Mu_1^2\ dv=\int_Xu_1^2\ dv +\int_Nu_1^2\ dv\leq c\int_M(\|u_1\|^2-\lambda u_1^2)\ dv
+c\epsilon\int_M(u^2+\|\nabla u\|^2)\ dv.
\end{equation}

We now observe that the constants `c' in (\ref{u0}) and (\ref{u1raw}) can be taken to be the same since the left hand sides are positive terms. Thus, the constants multiplying the first term in (\ref{u0}) and (\ref{u1}) are the same. 

We also observe that if we let
$L$ denote the operator $\Delta+\lambda$ and we consider smooth functions $\eta_1$ and
$\eta_2$ of compact support, the quadratic form $Q(\eta_1,\eta_2)=\int_M\eta_1L\eta_2\ dv$
is symmetric. Thus we may add (\ref{u0}) and (\ref{u1}) to obtain
\[ \int_M(u_0^2+u_1^2)\ dv\leq c(Q(u_0,u_0)-Q(u_1,u_1))+c\epsilon\int_M(u^2+\|\nabla u\|^2)\ dv.
\]
Since $Q$ is symmetric we have $Q(u_0,u_0)-Q(u_1,u_1)=Q(u_0-u_1,u_0+u_1)$ and so
\[ \int_M u^2\ dv\leq c\int_M(u_0-u_1)Lu\ dv+c\epsilon\int_M(u^2+\|\nabla u\|^2)\ dv.
\]
We may rewrite the term
\[ \int_M\|\nabla u\|^2\ dv=-\int_MuLu\ dv+\lambda u^2,
\]
so we finally have
\[ \int_M u^2\ dv\leq c\int_M(u_0-u_1)Lu\ dv+c|\int_MuLu|+c\epsilon\int_Mu^2\ dv.
\]
Applying the Schwarz inequality and using $\|u_0-u_1\|^2\leq 2(\|u_0\|^2+\|u_1\|^2)=2\|u\|^2$ yield,
\[ \int_Mu^2\leq c\int_M(Lu)^2\ dv,
\]
and we have shown that $\lambda$ is not in the spectrum of $M$.
\end{proof}

\section{\textbf{Proof of Main Theorems}}

Here we prove the following general theorem. The theorems stated in the introduction are special cases. 
\begin{theorem}\label{second_main} Assume that $M$ is of class 
$\mathcal M(\bar{\rho},\bar{\delta},\Lambda)$ for $n\geq 4$ or of either class $\mathcal M_0(\bar{\rho},\bar{\delta},\Lambda)$ or of class $\mathcal M_1(\bar{\rho},\bar{\delta},\Lambda)$ for $n>1$. Given any integer $G$, there exists $\epsilon>0$
such that if $\bar{\rho}<\epsilon$, $\bar{\delta}<\epsilon$, and $\Lambda>\epsilon^{-1}$
then the spectrum of $M$
has at least $G$ gaps. 
If 
$M$ of class $\mathcal M$ or $\mathcal M_1$, and 
the number of connected components of $X$ is infinite, then the spectrum can be replaced by the essential spectrum. 
If 
$M$ is of class $\mathcal M_0$ then we can replace spectrum by essential spectrum provided that the number of connected components of $X$ modeled on $X_\alpha$ is infinite. Here $X_\alpha$ has an infinite number of eigenvalues with eigenfunctions which
vanish at the centers of all balls that are removed to form connected components of $X$. 
\end{theorem}

To complete the proof of the existence of gaps in the spectrum we must also show that
if we choose $\lambda_1<\lambda_2$ which lie in different connected components of
$\mathbb R\setminus\mathcal S$ (for $\mathcal M_0(\bar{\rho},\bar{\delta},\Lambda)$ replace by 
$\mathbb R\setminus\mathcal S_0$), 
then the interval $(\lambda_1,\lambda_2)$ has nonempty
intersection with the essential spectrum of $M$. This result together with Theorem \ref{first_main} implies our second main theorem on the existence of arbitrarily many spectral gaps.

We will need a preliminary lemma which characterizes the essential spectrum.
\begin{lemma}(Donnelly, \cite[Prop. 2.2]{donnelly81})\label{interval} An interval $(\lambda-\epsilon, \lambda+\epsilon)$ intersects the essential spectrum if and only if there exists an infinite dimensional vector subspace $G_\epsilon$ of the domain $\mathfrak{D}(\Delta)$ such that for every $f\in G_\epsilon$, we have,
\[||\Delta f+\lambda f||_{L^2(M)}\leq \epsilon ||f||_{L^2(M)}.\]
\end{lemma}
We first assume that
$\lambda_1<\lambda_2$ where $\lambda_1$ and $\lambda_2$ are not in $\mathcal S$. By Theorem \ref{first_main}, $\lambda_1$ and $\lambda_2$ are not in the spectrum of $M$ provided $\bar{\rho}$
and $\bar{\delta}$ are small enough and $\Lambda$ is large enough. 

To show that the interval
contains points of the essential spectrum we choose a number $\lambda\in\mathcal S\cap (\lambda_1,\lambda_2)$. In particular, there is an $\epsilon_0>0$ such that 
the interval $(\lambda-\epsilon_0,\lambda+\epsilon_0)$ is contained in $(\lambda_1,\lambda_2)$.

From Lemma \ref{interval} it follows that the interval $(\lambda-\epsilon_0,\lambda+\epsilon_0)$
intersects the essential spectrum provided we can find an infinite dimensional space of
smooth compactly supported functions $u$ satisfying
\[ \int_M(\Delta u+\lambda u)^2\ dv<\epsilon_0^2\int_Mu^2\ dv.
\]

Let $\lambda\in\mathcal S\cap(\lambda_1,\lambda_2)$ and thus $\lambda$ is an eigenvalue of 
$X_\alpha$ for some $\alpha\in\{1,\ldots, p\}$. Let $v$ be an eigenfunction for $\lambda$. It is observed that there are pointwise bounds (which depend on the geometry of $\mathcal X$) on $v$ and its derivatives in terms of the $L^2$ norm. 
We will let $u=\zeta v$ where $\zeta$ is a cutoff function near $\partial \hat{X}$
where $\hat{X}$ is a connected component of $X$ modeled on $X_\alpha$. 

Next for each class of manifolds, the argument will vary slightly. For clarity, we'll state the results separately. 
\begin{prop}\label{exist_gap1} Assume that $M^n$ is of class 
$\mathcal M(\bar{\rho},\bar{\delta},\Lambda)$ and $n\geq 4$. 
Let $\lambda_1<\lambda_2$ and both lie in $\mathbb R\setminus\mathcal S$. 
Suppose there is an eigenvalue $\lambda\in \mathcal S\cap (\lambda_1,\lambda_2)$ 
of some $X_\alpha$ and $X_\alpha$
occurs infinitely often among the components of $X$. If $\bar{\rho}$ and $\bar{\delta}$ are chosen small enough, then the essential spectrum of
$M$ has nontrivial intersection with $(\lambda_1,\lambda_2)$.
\end{prop}

\begin{proof} 

For 
$n\geq 5$ we can choose for each boundary component of the form $\partial B_\rho$
the function $\zeta$ which is one outside $B_{2\rho}$ and zero near $\partial B_\rho$
so that $\zeta+\rho|\nabla\zeta|+\rho^2|\nabla\nabla \zeta|$ is bounded. Since we have
pointwise bounds on $v$ and its derivatives in terms of the $L^2$ norm, we have
\[ \int_{\hat{X}}(\Delta u+\lambda u)^2\ dv\leq c(\int_{\hat{X}}\zeta^2+|\nabla\zeta|^2+|\nabla\nabla\zeta|^2)\ dv)\int_{\hat{X}}u^2\ dv.
\]
This implies
\[ \int_M(\Delta u+\lambda u)^2\ dv\leq c\bar{\rho}^{n-4}\int_Mu^2\ dv
\]
which gives the desired result if $\bar{\rho}$ is small enough.

For $n=4$ a modification of the above argument works where we choose $\zeta$
near each boundary component to be a linear function of $\log(r)$ which is $1$
at $r=\sqrt{\rho}$ and $0$ at $r=\rho$ (see also Lemma \ref{logcutoff}). We then have
\[ r|\nabla\zeta|+r^2|\nabla\nabla\zeta|\leq c|\log(\rho)|^{-1},
\]
and so as above
\[ \int_{\hat{X}}(\Delta u+\lambda u)^2\ dv\leq c(\int_{\hat{X}}\zeta^2+|\nabla\zeta|^2+|\nabla\nabla\zeta|^2)\ dv)\int_{\hat{X}}u^2\ dv.
\]
Now this implies by easy estimation
\[ \int_M(\Delta u+\lambda u)^2\ dv\leq c|\log(\bar{\rho})|^{-1}\int_Mu^2\ dv
\] 
and again we have the desired result if $\bar{\rho}$ is small enough.
\end{proof}
\begin{prop}\label{exist_gap2} Assume that $M^n$ is of class 
$\mathcal M_0(\bar{\rho},\bar{\delta},\Lambda)$. 
Let $\lambda_1<\lambda_2$ and both lie in $\mathbb R\setminus\mathcal S_0$. 
Suppose there is an eigenvalue $\lambda\in \mathcal S_0\cap (\lambda_1,\lambda_2)$ of some $X_\alpha$ and $X_\alpha$
occurs infinitely often among the components of $X$. If $\bar{\rho}$ and $\bar{\delta}$ are chosen
small enough then the essential spectrum of
$M$ has nontrivial intersection with $(\lambda_1,\lambda_2)$.
\end{prop}

\begin{proof}
For $n\geq 3$ and $M$ of class $\mathcal M_0$, we take an $X_\alpha$ with an eigenvalue
$\lambda\in(\lambda_1,\lambda_2)$ and an eigenfunction $v$ which vanishes at the centers
of the boundary spheres of $\hat{X}$. In this case we have for $x\in B_{\sqrt{\rho}}$
\[ v(x)^2\leq cr^2\sup_{B_{\sqrt{\rho}}}|\nabla v|^2\leq cr^2\int_{\hat{X}}u^2.
\]
Therefore we have on $B_{\sqrt{\rho}}$, 
\[ (\Delta u+\lambda u)^2\leq c(|\nabla\zeta|^2+r^2|\nabla\nabla\zeta|^2)\int_M u^2.
\]
We can choose $\zeta$ to cutoff on $B_{2\rho}\setminus B_\rho$ and get
\begin{align*}
\int_M(\Delta u+\lambda u)^2\ dv&\leq c\rho^{-4}\int_{B_{2\rho}\setminus B_{\rho}}r^2\ dv\int_Mu^2\ dv,\\
&\leq c\rho^{-4}\rho^{n-1}\rho^{3}=c\rho^{n-2}\leq c\bar{\rho}\int_Mu^2\ dv.
\end{align*}
That completes the proof if $\bar{\rho}$ is small.\\

In case $n=2$ we must use the logarithmic cutoff function as we did for $n=4$ 
above and as in that case we get
\[ \int_M(\Delta u+\lambda u)^2\ dv\leq c|\log(\bar{\rho})|^{-1}\int_Mu^2\ dv.
\]
\end{proof}

\begin{prop}\label{exist_gap3} Assume that $M^n$ is of class 
$\mathcal M_1(\bar{\rho},\bar{\delta},\Lambda)$. 
Let $\lambda_1<\lambda_2$ such that both lie in $\mathbb R\setminus\mathcal S$. 
Furthermore, suppose there is an eigenvalue $\lambda\in\mathcal S\cap (\lambda_1,\lambda_2)$ of some $X_\alpha$ and $X_\alpha$
occurs infinitely often among the components of $X$. If $\bar{\rho}$ and $\bar{\delta}$ are chosen
small enough then the essential spectrum of
$M$ has nontrivial intersection with $(\lambda_1,\lambda_2)$.
\end{prop}
\begin{proof}
The idea is to localize our eigenfunctions
by cutting off `across the neck'. We consider a connected component $\hat{N}$
whose boundary consists of spheres in a bounded number of components $\hat{X}$.
Let $\lambda\in(\lambda_1,\lambda_2)$ be an eigenvalue of $X_\alpha$ and $v$
an eigenfunction. 

We first modify $v$ to a function $v_1$ which is constant near the
centers of the balls which are removed. We can do this by setting $v_1=v-\zeta w$
where $w=v-v(p)$ with $p$ the center point of the ball and $\zeta$ a cut-off function being $1$ around $p$. Because $w$ is zero at
the center, for a small $\epsilon_0$, we can choose $\zeta$ as in Prop \ref{exist_gap2} to have
\[ \int_{\hat{X}_1}(\Delta v_1+\lambda v_1)^2\ dv\leq \epsilon_0\int_{\hat{X}_1}v_1^2\ dv
\]
where $\hat{X}_1 \subset \hat{X}$ is $X_\alpha$ with balls of fixed radius $\rho_1$ depending on $\epsilon$ removed. Also, $\rho_1$ could be chosen arbitrarily small as that only improves the inequality.

We are then free to require $\bar{\rho}<<\rho_1$. 
We consider the corresponding enlarged neck component $\hat{N}_1$ such that $\hat{X}\setminus \hat{N}_1=\hat{X}_1$.
Thus $\hat{N}_1$ has a finite
number of boundary spheres of a fixed radius. 

We then solve the Dirichlet problem
to obtain a harmonic function $h$ on $\hat{N}_1$ which is equal to $v(p)$ 
on $r=\rho_1$ (in the core component we are considering) and equal to $0$ on all other boundary components. Since $h$ minimizes the Dirichlet integral, 
we may compare
the Dirichlet integral of $h$ to that of a function which is equal $v(p)$ on $\partial B_{\rho_1}$
and equal to $0$ on $\partial B_\rho$ and $0$ on the rest of $\hat{N}_1$. We have seen
that for $n\geq 2$ we may choose such a function with small Dirichlet integral (for $n=2$
we use a linear function of $\log(|x|)$). In fact the Dirichlet integral may be made
arbitrarily small on the order $O(|\log(\bar{\rho})|^{-1})$.

We may now use
elliptic boundary estimates to show that the $C^2$ norm of $h$ is bounded by a constant $O(1)$
in the annulus $B_{\rho_1}\setminus B_{\rho_1/2}$. 

\textit{Claim: In $B_{\rho_1}\setminus B_{3\rho_1/4}$, $||\nabla h||_{L^\infty}\leq (||\nabla h||_{L^2})^\tau=D^\tau$ for $\tau=\frac{1}{2n}$.}
 
Suppose not; then the value of $\|\nabla h\|$, in a ball of radius $\frac{D^\tau}{2C}$, is at least $ \frac{D^\tau}{2}$ by the $C^2$ estimate. The volume of such a neighborhood is at least $CD^{n\tau}$ (here $D^\tau\leq c |\log(\bar{\rho})|^{-\tau}<\rho_1/2$). Integrating over this neighborhood yields a contradiction as $D$ is small and $D^{\tau+n\tau}> D$. So the claim is proved. \\   
 
Thus the first derivative is small and
the function is close to $v(p)$ in this annulus. Elliptic boundary estimates applied
to $h-v(p)$ now imply smallness of all derivatives near $\partial B_{\rho_1}$. It follows
similarly that all derivatives of $h$ are small near the other boundary components of
$\hat{N}_1$.

We can therefore take $u$ to be a smoothed version of the function which is $v_1$
outside $B_{\rho_1}$, $h$ in $\hat{N}_1$, and $0$ on the remainder of $M$. We construct
such a function for each eigenfunction of $X_\alpha$ with eigenvalue $\lambda$ and
for each occurrence of $X_\alpha$ among the connected components of $X$ (infinitely
many by assumption). Thus we take the linear span and produce an infinite dimensional
space $E_0$ of functions which are eigenfunctions outside balls on the core components
and which are harmonic on the enlarged components of $N$. For a function $u\in E_0$
we then have on each reduced core region $\hat{X}_1$
\[ \int_{\hat{X}_1}(\Delta u+\lambda u)^2\ dv\leq \epsilon\int_{\hat{X}_1}u^2\ dv
\]
since $u$ is an eigenfunction with eigenvalue $\lambda$ for $X_\alpha$ on $\hat{X}_1$
and the metric $g$ is $C^1$ close to $g_\alpha$ on $X_1$. Now on each enlarged
neck region $\hat{N}_1$ we have
\[ \int_{\hat{N}_1}(\Delta u+\lambda u)^2\ dv\leq \lambda Vol(\hat{N}_1)\sup_{\hat{N}_1}u^2
\] 
since $u$ is harmonic on these regions. Now the term on the right is bounded by the
maximum value of $u^2$ on each boundary component. This in turn is the squared value of an
eigenfunction at a point of $X_\alpha$ and thus may be estimated by the the square of the
$L^2$ integral over the corresponding core component of $X$. Also the volume 
of $\hat{N}_1$ bounded by $\bar{\delta}+ C(\rho_1)^n$ is small so we have
\[ \int_{\hat{N}_1}(\Delta u+\lambda u)^2\ dv\leq \epsilon\int_{\hat{X}}u^2\ dv.
\]
Combining with the estimate on $\hat{X}_1$ and summing over to obtain
\[ \int_M(\Delta u+\lambda u)^2\ dv\leq 2\epsilon\int_M u^2\ dv
\]
for any $u\in E_0$. Since $\epsilon$ can be taken arbitrarily small by choosing $\bar{\rho}$ and $\rho_1\approx c|\log(\bar{\rho})|^{-\frac{1}{2n}}$
small enough, we have completed the proof.
\end{proof}
\begin{remark}
The assumption of infinite occurrence is vital because of the decomposition principle for essential spectrum (stable under compact perturbations).
\end{remark}

Now we collect proofs of our main theorems.
\begin{proof} (\textbf{Theorem \ref{second_main}})
Since the number of connected components of X is infinite, there must be some $X_\alpha\in \mathcal X$ which occurs infinitely often. Let $0=\lambda_0<\lambda_1<...$ be distinct eigenvalues of 
$X_\alpha$. Let $D=\lambda_{G+1}$. 

Then, for $0\leq k\leq G$ each $[\lambda_k,\lambda_{k+1}] \cap \mathcal S$ has at least one interval. Consequently, we can choose $G$ compact intervals, each of distance at least $d$ from $\mathcal S$. By Theorem \ref{first_main} and Lemma \ref{interval}, for each interval, we can choose small $\bar{\rho}$ and $\bar{\delta}$ (only depending on $d$, $D$, and the geometry of $\mathcal X$) such that the interval is disjoint from the essential spectrum.

On the other hand, by Propositions \ref{exist_gap1}, \ref{exist_gap2}, and \ref{exist_gap3}, for $\epsilon_0<d$ and $0\leq k\leq G$, we can choose small $\bar{\rho}$ and $\bar{\delta}$ such that $(\lambda_k-\epsilon_0,\lambda_k+\epsilon_0)$ intersects non-trivially with the essential spectrum. 

Taking the minimum values of $\bar{\rho},\bar{\delta}$ above completes the proof.      
\end{proof}
\begin{remark}
Colin de Verdi{\`e}re \cite{yves87} showed that there exists a compact manifold with finitely prescribed eigenvalues. Using that as a model, our methods can be used to give prescribed gap intervals for the appropriate essential spectrum. Such a result was obtained for metrics on a torus by
Khrabustovskyi \cite{khra12}. 
\end{remark}

The theorems from the Introduction now follow immediately.
\begin{proof} (\textbf{Theorem \ref{intro1} and Theorem \ref{intro2}}) These follow from applying Theorem \ref{second_main} to Prop. \ref{covering_gaps} and Prop. \ref{complete} respectively.
\end{proof}

\def\cprime{$'$}
\bibliographystyle{plain}
\bibliography{bio}

\end{document}